\documentclass[12pt]{article} 
\usepackage{latexsym,amsfonts,amssymb,amsmath,mathrsfs,amscd} 
\newtheorem{thm}{Theorem}[section] 
\newtheorem{crl}[thm]{Corollary} 
\newtheorem{prp}[thm]{Proposition} 
\newtheorem{lm}[thm]{Lemma} 
 
\newtheorem{ex}[thm]{Example} 
\newtheorem{dfn}[thm]{Definition} 
\newenvironment{proof}{\noindent{\bf Proof.}}{\noindent$\Box$\par\medskip}

\newcommand{\A}{{\mathcal A}}
\newcommand{\B}{{\mathcal B}}
\newcommand{\M}{{\mathcal M}}

\newcommand{\R}{{\mathbb R}}
\newcommand{\C}{{\mathbb C}}
\newcommand{\lie}{{\mathcal L}}
\newcommand{\F}{{\mathcal F}}
\newcommand{\g}{{\mathfrak g}}

\newcommand{\supp}{\mathrm {supp}}

\newcommand{\im}{\mathrm {im}}

\newcommand{\la}{\mathrm {L}}
\newcommand{\os}{\mathrm {OS}}

\newcommand{\sudda}[1]{}

\newcommand{\gr}{\mathrm {gr}}

\newcommand{\ext}{\mathrm {Ext}} 
\newcommand{\tor}{\mathrm {Tor}} 
\newcommand{\U}{\mathrm {U}} 
\newcommand{\rank}{\mathrm {rank}}  
\newcommand{\id}{\mathrm {id}}
\newcommand{\subalg}{\mathrm {subalg}} 

\begin{document}
\date{} 
\title{The holonomy Lie algebra of a matroid}   
\author {Clas L\"ofwall}
\maketitle
\centerline{Seminar in Stockholm 9 dec 2015}
\section{Introduction}
\vspace{30pt}
I will soon start from the beginning and define a matroid and its Orlik-Solomon algebra and holonomy Lie algebra, but first I will give some background from topology and cohomology. A (central) hyperplane arrangement is a finite number of subspaces of codimension one in a finite dimensional vector space over $\R$, $\C$ or even finite fields. The space $X$ which is the complement of the union of the hyperplanes has interesting properties. If you consider $\mathbb {R}^n$ the cohomology is however rather trivial, there is just an $H^0$, since $X$ is a disjoint union of contractible sets. But it is not a trivial task to compute $|H^0|$. The way to do it is to consider the intersection lattice of the hyperplanes and compute its M\"obius function. The result is that the number of components is the sum of the absolut values of the M\"obius function (Zaslavsky 1975, \cite{Za}). 

When we change to $\mathbb{C}^n$ we get something quite different. If we think of $\mathbb{C}^n$ as $\mathbb{R}^{2n}$ the $\mathbb{C}$-subspaces of codimension one now has real codimension two and the simplest case is the real plane minus the origin, whose first cohomology group is non-zero. In fact, the general complex case may be treated with deRham cohomology using a closed but not exact one-form for each hyperplane to get the cohomology algebra (at least over $\mathbb{R}$). This is called the Orlik-Solomon algebra \cite{orl} (Brieskorn \cite{bri} showed that the cohomology has no torsion), which we will come back to. The algebra is determined by knowing which subsets of the defining forms of the hyperplanes are independent. It is here matroids come in; the main example of a matroid is the set of independent subsets of a set of vectors. Goresky-MacPherson \cite{gor}, Bj\"orner-Ziegler \cite{bz} and  Longueville and Schultz \cite{lon} study general subspace arrangements in $\mathbb {R}^{2n}$, especially subspaces of codimension 2 whose intersections have even dimension. This is the natural first step generalization of the complex case. They prove that everything hold except that the signs in the differential defining the Orlik-Solomon algebra may be different from the complex case. This is interesting, since it gives us more freedom in defining holonomy Lie algebras (and my generalization covers also this case). I should perhaps end this introduction by defining the holonomy Lie algebra (it will be done more explicitly later with generators and relations). If $\mathcal O$ is the cohomology algebra of $X$, with coefficients in some field $k$, of a complex arrangement, we have that $\ext_\mathcal O(k,k)$ is the enveloping algebra of the "homotopy" Lie algebra $\g$ of the arrangement. The holonomy Lie algebra $\eta$ of the arrangement is defined as the Lie subalgebra of $\g$ generated by its elements of degree one. The following result by Kohno \cite{koh} highly increases the interest to study $\eta$ ($k$ has characteristic zero).

$$
\eta\cong\gr(\pi_1(X))\otimes k 
$$
where $\gr(\pi_1(X))$ is the graded associative with respect to the lower central series of $\pi_1(X)$.
 
 \section{Matroids}
 Let $E$ be a finite set. A subset $\M$ of $\mathscr P(E)$ is a matroid if the following two axioms are fulfilled. 
 \begin{align*}
 1)&\quad A\in\M, \ B\subset A\quad \Rightarrow \quad B\in\M\\
 2)&\quad A,B\in\M, \ |A|<|B|\quad\Rightarrow \quad\exists x\in B\setminus A \text{ such that }A\cup\{x\}\in\M
 \end{align*}
 The axioms are obviously fulfilled if $E$ is a set of vectors in a vector space and $\M$ is the set of independent subsets of $E$. A matroid is called representable if it is given in this way. We will use  ``$A$ independent" and ``$A$ dependent" for the cases $A\in\M$ and $A\notin\M$ respectively. We will always assume that a matroid is "simple", which means that all subsets of size at most two are independent. This is true for the case when $E$ is a set of defining forms for the hyperplanes in a central arrangement. Indeed, two different hyperplanes through origin have non-parallel forms. A graph defines a matroid on its set of edges. A set of edges is independent if it contains no cycles. The matroid is simple if there are no loops and no multiple edges. A graph matroid is representable. If the vertices are $\{1,\ldots,n\}$ we may define vectors in $\R^n$ by defining for each edge between $i$ and $j$ the vector $e_i-e_j$ (corresponding to the hyperplanes $x_i-x_j=0$). The complete graph $K_n$ on $n$ vertices has an interesting holonomy Lie algebra, obtained by Kohno \cite{koh}, which we will come back to. The corresponding arrangement is called the "braid" arrangement, since its fundamental group is known to be the pure braid group. 
\subsection{Rank, bases, flats}
Let $T\subseteq E$, $\M$ a matroid on $E$. An independent subset $B$ of $T$ is called a basis for $T$ if $B$ is maximal. All bases for $T$ have the same size, by the exchange property 2). The size of a basis for $T$ is called $\rank(T)$. If $B$ is a basis for $T$, then $B\cup \{x\}$ is dependent for all $x\in T\setminus B$. A $k$-$flat$ $F$ is a maximal subset of $E$ of rank $k$. Then $F=B\cup\{x\in E;\ B\cup\{x\}\text{ dependent}\}$ where $B$ is a basis for $F$. Moreover if $B$ is independent, then $B\cup\{x\in E;\ B\cup\{x\}\text{ dependent}\}$ is a $|B|$-flat.   
\begin{lm} If $F_1$ and $F_2$ are flats, then $F_1\cap F_2$ is a flat of rank less than $\min(\rank(F_1),\rank(F_2))$. 
\end{lm}
\begin{proof} Let $B$ be a basis for $F_1\cap F_2$ and extend $B$ to bases $B_1$, $B_2$ for $F_1$, $F_2$ respectively. If $B\cup\{x\}$ is dependent, then $B_i\cup \{x\}$ is dependent for $i=1,2$. Hence $x\in F_1\cap F_2$. If $x\in F_1\cap F_2\setminus B$ then $B\cup \{x\}$ is dependent. Hence, by the above, $F_1\cap F_2$ is a flat.
\end{proof}
We may now define a closure operator $\bar T$ as the least flat that contains $T$. All flats form a lattice $\la(\M)$ under inclusion, which is graded by rank. If $\M$ is the matroid of a hyperplane arrangement, then $\la(\M)$ is isomorphic to the lattice of all intersections of the hyperplanes ordered by reverse inclusion, since
$$
\bigcap_{i\in T_1}H_i\subseteq\bigcap_{i\in T_2}H_i\Leftrightarrow\overline{\{n_i;i\in T_1\}}\supseteq\overline{\{n_i;i\in T_2\}}
$$
where $n_i$ is a defining form of $H_i$. 

The 2-flats are of special interest for us, since they determine the holo\-nomy Lie algebra. Since the intersection of two 2-flats is a 1-flat or a 0-flat, it follows that two different 2-flats have at most one element in common. Also all pairs are contained in a (unique) 2-flat. This is called a 2-partition of $E$ in \cite{ox2}, where it is also proven that every 2-partition is the set of 2-flats of a matroid of rank $\le3$ (the independent 3-sets are the 3-sets that are not contained in any of the sets in the 2-partition), se \cite{ox2}, Prop. 2.1.21 or \cite{lof}. 
\begin{ex}$E=\{1,2,3,4,5,6,7\}$. A 2-partition on $E$ is 
$$
\{\{1,2,3,4\},\{1,5,6\},\{2,5,7\},\{3,6,7\},\{1,7\},\{2,6\},\{3,5\},\{4,5\},\{4,6\},\{4,7\}\}
$$
\end{ex}
The matroid of rank 3 determined by the 2-partition in this example has the following dependent 3-sets. 
$$
\{1,2,3\},\{1,2,4\},\{2,3,4\},\{1,5,6\},\{2,5,7\},\{3,6,7\}
$$ 
Usually one does not give the 2-sets in the partition, since they are determined by the rest. We will use the following terminology.
\begin{dfn}
A set-arrangement on a finite set $E$ is a set $\A$ of subsets of $E$ such that $|A|\ge3$ for all $A\in\A$ and $|A\cap B|\le1$ for all $A,B\in\A$.
\end{dfn}
A set-arrangement (which may be empty) on $E$ determines a 2-partition on $E$ and vice versa.
Observe that for any matroid $\M$ we have that $\M_{\le3}=\{A\in\M;\ \rank(A)\le3\}$ is a matroid with the same set of 2-flats as $\M$.

\section{The Orlik-Solomon algebra and the holo\-nomy Lie algebra}
The Orlik-Solomon algebra on a simple matroid $\M$ on $\{e_1,\ldots,e_n\}$ is 
$$
\os(\M)=E(e_1,\ldots,e_n)/\langle\partial e_S\rangle
$$
where $\{e_1,\ldots,e_n\}$ is the exterior algebra, $\partial e_i=1$, $\partial$ is a derivation, $S=\{e_{i_1},\ldots,e_{i_k}\}$ is dependent and $e_S=e_{i_1}\cdots e_{i_k}$,  and $\langle\partial e_S\rangle$ is the ideal generated by $\{\partial e_S;\ S \text{ dependent}\}$. The algebra is graded by letting the degree of $e_i$ be 1 for all $i$. The Hilbert series $\os(\M)(z)$, i.e., the generating series for the dimensions of the different graded pieces is obtained from the M\"obius function of the lattice $\la(\M)$, e.g., if $\rank(\M)=3$, then $\os(\M)(z)=1+nz+dz^2+(d+1-n)z^3$, where $d=\sum(|A|-1)$ and the sum is taken over all 2-flats $A$. Observe that $e_S=0$ in $\os(\M)$ if $S$ is dependent, since $e_S=e_{i_1}\partial e_S$, whence $\os(\M)_i=0$ if $i>\rank(\M)$, which also follows from the fact that $\la(\M)$ vanishes above $\rank(\M)$. (Also, as was mentioned in the introduction, the number of components of a hyperplane arrangement in $\R^n$ can be obtained by evaluating the Hilbert series $\os(\M)(z)$ at $z=1$.)

The quadratic relations are $ab-ac+bc$ for all $\{a,b,c\}$ which are dependent, i.e., are contained in a 2-flat. In order to define the holonomy Lie algebra we want to compute the orthogonal complement of these quadratic relations. In this process it is enough to consider one 2-flat at a time, since a quadratic monomial occurs only in one 2-flat. So it is enough to consider a matroid $\M_2$ of rank 2 on $\{1,\ldots,n\}$, where $\{1,\ldots,n\}$ is the only 2-flat ($\M_2$ is realized as $n$ lines in $\C^2$). Write $\os(\M_2)$ as a quotient of a free non-commutative algebra. Here is a (triangular) independent set of quadratic relations: $e_1^2,\ldots,e_n^2,\ e_je_i+e_ie_j$ for $j>i$, followed by the relations defined by the triples $123,\ 134,\ 145,\ldots, 1(n-1)n,\ 234,\ldots, 2(n-1)n,\ldots, (n-2)(n-1)n$. In total this gives $n+{n\choose 2}+{n-1\choose2}=n^2-(n-1)$ independent relations. Hence the orthogonal complement is of dimension at most $n-1$. Let $x_1,\ldots,x_n$ be a dual basis of $e_1,\ldots,e_n$. An element in the orthogonal complement is a linear combination of $[x_i,x_j]=x_ix_j-x_jx_i$. Here is a set of solutions:
\begin{align*}
&[e_1,e_1]+[e_1,e_2]+\cdots+[e_1,e_n]\\
&[e_3,e_1]+[e_3,e_2]+\cdots+[e_3,e_n]\\
&\vdots\\
&[e_n,e_1]+[e_n,e_2]+\cdots+[e_n,e_{n}]\\
\end{align*} 
In a suitable order, the leading monomials are $e_1e_2,e_1e_3,\ldots,e_1e_n$ and hence this gives $n-1$ independent elements in the orthogonal complement, so it is in fact a basis. Summing the above elements, we get also $[e_2,e_1]+[e_2,e_2]+\cdots+[e_2,e_n]$. Letting $c=\sum_{i=1}^ne_i$ we get that the holonomy Lie algebra of the matroid of rank 2 on $\{e_1,\ldots e_n\}$ is 
$$
\eta=\F(e_1,\ldots,e_n)/\langle[e_i,c]; i=1,\ldots,n\rangle
$$
where $\F(e_1,\ldots,e_n)$ is the free Lie algebra on $e_1,\ldots,e_n$. Hence $\eta/\langle c\rangle$ is free on $n-1$ generators and $\eta$ is the direct product of a free Lie algebra on $n-1$ generator and a free Lie algebra on one generator. The holonomy Lie algebra of an arbitrary simple matroid on $E=\{e_1,\ldots,e_n\}$ (or: of a set-arrangement on $E$) is the free Lie algebra on $e_1,\ldots,e_n$ modulo relations as above for each 2-flat with at least three elements (or: for each set in the set-arrangement). The variables in a 2-flat with two elements commute (or: the pair of variables which are not contained in any set of the set-arrangement commute). The question is how these "local" Lie algebras defined from the 2-flats are glued together. 
\section{Structure of the holonomy Lie algebra}
Much of the theory for the structure of the full holonomy Lie algebra in terms of the local ones does not depend on how the relations look like. Hence we assume that we are given a set-arrangement $\A$ on $E=\{x_1,\ldots,x_n\}$ with local Lie algebras $L_A=\F(A)/\langle R_A\rangle$, where $R_A\subset\F(A)'=[\F(A),\F(A)]$, for each $A\in\A$. Define $\lie=\F(E)/\langle R\rangle$, where 
$$
R=\cup_{A\in\A} R_A\cup\{[x_i,x_j];\ \{x_i,x_j\} \text{ is not a subset of any } A\in\A\}
$$

Define a Lie homomorphism $\pi_A:\ \F(E)\to L_A$ by sending $x\in E\setminus A$ to zero. Since any Lie-monomial in $R\setminus R_A$ contains at least one generator $\notin A$, the map factors and defines a surjection $\pi_A:\ \lie{}\to L_A$. There is also a natural map $s_A:\ L_A\to \lie{}$ and $\pi_A\circ s_A=\id_{L_A}$ for each $A\in\A$. Hence we have a split exact sequence of Lie algebras for each $A\in\A$,
\begin{equation}\label{eq1}
0\to I_A\to \lie{}\to L_A\to0
\end{equation}
Now consider the derived Lie algebras, $\lie{}'$, $L_A'$ and the Lie homomorphisms $\pi_A',s_A'$. We have that $\pi_B'\circ s_A'=0$ for $B\neq A$. There is a surjective Lie homomorphism $\pi=\oplus\pi_A':\lie{}'\to \oplus L_A'$, which has a section $s=\sum s_A'$ that is a module homomorphism (not Lie in general). Hence we have an exact sequence of Lie algebras
\begin{equation}\label{eq2}
0\to I\to \lie{}'\to\oplus_{A\in\A}L_A'\to0
\end{equation} 
and $\sum \im(s_A')\subset \lie{}'$ is a direct sum, since $s$ is injective.
\begin{ex} In {\rm \cite{roo}} the following set-arrangement is studied
$$
(124),\ (135),\ (236),\ (167),\ (258),\ (457),\ (468)
$$
and it is (almost) proved that the ideal $I$ in {\rm (\ref{eq2})} is one-dimensional in each degree $\ge3$.
\end{ex}
\begin{prp}\label{prp}
$$
I=0\ \Leftrightarrow \ [x,\im(s_A')]=0 \text{ for all }A \text{ and all }x\in E\setminus A
$$
\end{prp}
\begin{proof}
\begin{align*}
&I=0\Rightarrow\ker(\oplus \pi_B')=0\Rightarrow[x,\im(s_A')]=0, x\notin A\Rightarrow
\\
&\sum\im(s_A')\text{ is an ideal containing }\lie{}'\Rightarrow 
\\
&s\text{ surjective }\Rightarrow s\text{ iso }\Rightarrow\pi\text{ mono }\Rightarrow I=0
\end{align*}
\end{proof}
\begin{dfn} A Lie algebra $\lie{}$ of a set-arrangement $\A$ is said to be decomposable if $I=0$.
\end{dfn}
In order to get more precise results, we introduce a condition on the local Lie algebras which is satisfied for the holonomy Lie algebra.
\begin{dfn} A Lie algebra $\lie{}$ of a set-arrangement $\A$ is said to satisfy {\rm replacement} if for any $A\in\A$ and any $x\in A$ such that also $x\in B$ for some $B\in\A$, $B\ne A$, it holds that $L_A'\subseteq\subalg_{L_A}(A\setminus\{x\})$.
\end{dfn}
Here $\subalg_L(X)$ denotes the Lie subalgebra of $L$ generated by the set $X\subseteq L$.

The following proposition is easily proved by induction.
\begin{prp}\label{prp2} If $\lie{}$ satisfies replacement, then 
\begin{itemize}
\item $I=0\Leftrightarrow [x,[y,z]]$ for all $A\in\A$ and all $x\in E\setminus A$ and all $y,z\in A$ 
\item $I_A=\subalg_\lie(E\setminus A)$ for all $A\in\A$
\end{itemize}
\end{prp}
The first part of the proposition is a result in \cite{pap} and is implemented in Macaulay2.

The formulas (\ref{eq1}) and (\ref{eq2}) may be generalized in the following way.
\begin{dfn}
Given a set-arrangement $\A$ on $E$, a subset $\B$ of $\A$ is {\rm closed} in $\A$ if for all $A\in\A\setminus\B$ it holds that $|A\cap \supp(\B)|\le1$ where $\supp(\B)=\cup_{B\in\B}B$. 
\end{dfn}
Let $\A$ be a set-arrangement on $E$ and $\B$ a subset of $\A$ which is closed in $\A$. Define $L_{\B}=\F(\supp(\B))/\langle R_{\B}\rangle$, where 
$$
R_{\B}=\cup_{B\in\B}R_B\cup\{[x,y];\ x,y\in\supp(\B) \text{ and }\forall B\in\B\ \{x,y\}\not\subseteq B\}
$$
It follows that $L_{\B}=L_A$, when $\B=\{A\}$ and given the definitions of the local Lie algebras $L_A$ for $A\in\A$, $L_{\B}$ is the Lie algebra of $\B$ viewed as a set-arrangement on $\supp(\B)$.

A map $\pi_{\B}: \lie\to L_{\B}$ is defined by sending the variables in $E\setminus \supp(\B)$ to zero. A natural map $s_{\B}: L_{\B}\to \lie$ may be defined since two different elements in $\supp(\B)$ is not contained in a single $A$, where $A\not\in\B$. Also, $\pi_{\B}\circ s_{\B}=\id_{L_{\B}}$. Hence we have a split exact sequence of Lie algebras
\begin{equation}\label{eq3}
0\to I_{\B}\to \lie{}\to L_{\B}\to0
\end{equation}
Now suppose we have disjoint closed subsets $\B_i,\ i=1,\ldots k$ of $\A$, such that $\cup_{i=1}^k\B_i=\A$. In the same way as above, we may define  a surjective Lie homomorphism $\pi=\oplus\pi_{\B_i}':\lie{}'\to \oplus L_{\B_i}'$, which has a section $s=\sum s_{\B_i}'$ that is a module homomorphism (not Lie in general). Hence we have an exact sequence of Lie algebras
\begin{equation}\label{eq4}
0\to J\to \lie{}'\to\oplus_{i=1}^kL_{\B_i}'\to0
\end{equation} 
and $\sum \im(s_{\B_i}')\subset \lie{}'$ is a direct sum.

\begin{prp} We have

\begin{itemize}
\item $J$ is generated as an ideal by $[x,\im(s_{\B_i}')]\text{ for all }i \text{ and all }x\in E\setminus \supp(\B_i)$
\item $J=0\ \Leftrightarrow \ [x,\im(s_{\B_i}')]=0 \text{ for all }i \text{ and all }x\in E\setminus \supp(\B_i)$
\end{itemize}

\end{prp}
 
The proof is the same as the proof of Proposition \ref{prp}. 
 We also have the following generalization of Proposition \ref{prp2}.
 \begin{prp}\label{prp3} If $\lie{}$ satisfies replacement, then 
\begin{itemize}
\item $J$ is generated as an ideal by $[x,[y,z]]\text{ for all }
y,z\in \supp(\B_i),\
x\in E\setminus\supp(\B_i),\ i=1,\ldots,k $
\item $J=0\Leftrightarrow [x,[y,z]]=0\text{ for all }
y,z\in \supp(\B_i),\
x\in E\setminus\supp(\B_i),\ i=1,\ldots,k $
\item $I_{\B_i}=\subalg_\lie(E\setminus \supp(\B_i))$ for all $i=1,\ldots k$
\end{itemize}
\end{prp}

\subsection{Application to graphs}
Paulo Lima-Filho and Hal Schenk \cite{lim} proved a general formula for the series of the holonomy Lie algebra of graph arrangements. We will now describe their technique. Let $G$ be a graph and $\A_G$ the corresponding set-arrangement on the edge set consisting of 2-flats of size at least three. A 2-flat has at most three elements, since 4 edges always contain 3 edges which is not a triangle and hence has rank 3. Hence $\A_G$ consists of all triangles in the graph. Let $L_G$ be the holonomy Lie algebra of $\A_G$. Hence $L_G$ is the free Lie algebra on the set of edges modulo all $[x,y]$ where $x,y$ is not part of a triangle and also for each triangle $(a,b,c)$ there are the relations $[a,b]-[b,c]$ and $[a,c]-[c,b]$. A subgraph of $G$ is closed in the meaning above if for all pairs $a,b$ of edges in the subgraph such that there is a triangle $(a,b,c)$ in $G$, it holds that also $c$ is in the subgraph.  

Let $A$ be a vertex and $V$ the subgraph of all edges going out from $v$. Let $K$ be the subgraph of edges between endpoints of the edges in $V$, let $G_1=V\cup K$ and $G_2=G\setminus V$. Then 
\begin{itemize}
\item $G_1$ and $G_2$ are closed in $G$.
\item $K$ is closed in $G_1$ and $G_2$.
\item If $K$ is complete, then $\ker(L_G\xrightarrow{\pi} L_{G_2})$ is free on the edges in $V$.
\end{itemize}
The first two claims are easily checked. We prove the last claim. We know that $\ker(\pi)=\subalg_{L_G}(V)$. Define an action of $L_{G_2}$ on $\F(V)$ by 
$$
x\circ v=
\begin{cases}
0&\text{if }x\notin K\\
[v,w]&\text{if }v,w\in V\text{ and }(x,v,w) \text{ is a triangle}
\end{cases} 
$$
We want to prove that the relations in $L_{G_2}$ operate as zero. If $[x,y]$ is a relation, then either one of $x,y$ is not in $K$, and then $[x,y]$ operates as zero or $x,y$ have no common vertex and then $x\circ v=0$ or $y\circ v=0$ for all $v\in V$. Now suppose $(x,y,z)$ is a triangel in $G_2$. If two of $x,y,z$ are not in $K$ then $[x,y],[y,z],[z,x]$ all operate as zero. If two of $x,y,z$ are in $K$, then also the third is. We can look at the subgraph consisting of four vertices $A,B,C,D$ and the four edges $v_1=AB,\ v_2=AC,\ v_3=AD$ and $x=BC,\ y=CD,\ z=BD$. A computation shows that $[x,y]\circ v_1=[y,z]\circ v_1=[z,x]\circ v_1=[v_1,[v_3,v_2]]$. It follows that there is a  semi-direct product of $L_{G_2}$ operating on $\F(V)$ which is a quotient of $L_G$ and at the same time this Lie algebra maps onto $L_G$. Hence they are isomorphic and the claim follows.
\begin{crl}{\rm(Kohno)} The holonomy Lie algebra of the complete graph $K_n$ is a successive semi-direct product of free Lie algebras on $1,2,\ldots,n-1$ generators and the series is $\Pi_{i=1}^{n-1}1/(1-it)$.
\end{crl}
\begin{crl}
{\rm (Papadima-Suciu)} The holonomy Lie algebra $L_G$ of a graph $G$ is not decomposable iff $G$ contains $K_4$ as a subgraph.
\end{crl}
\begin{proof} We know that $L_G$ is not decomposable iff there are three edges $x,y,z$ which is not a triangle such that $[x,[y,z]]\neq0$ in $L_G$. Suppose $K_4$ is a subgraph of $G$. By the last claim above we know that the three edges $x,y,z$ going from one vertex in $K_4$ satisfy $[x,[y,z]]\neq0$ in $L_{K_4}$. We prove by induction that $L_{K_4}$ is a Lie subalgebra of $L_G$. Suppose $A$ is a vertex in $G$ which is not in the subgraph $K_4$ and consider the subdivision of $G$ in $G_1,K,G_2$ as above. Since $G_2$ is closed in $G$, $L_{G_2}$ is a Lie subalgebra of $L_G$. Since $K_4$ is a subgraph of $G_2$ it follows by induction that $L_{K_4}$ is a Lie subalgebra of $L_{G_2}$ and hence of $L_G$ which proves one direction of the claim. Suppose there are three edges $x,y,z$ in $G$ which is not a triangle such that $[x,[y,z]]\neq0$ in $L_G$. Then $[y,z]\neq0$ and hence there is an edge $u\ne x$ such that $y,z,u$ is a triangle and $[x,[y,z]]=[x,[z,u]]=[x,[u,y]]$. Let the vertices of the triangle be $A,B,C$, where $y=AB,z=AC,u=BC$. The edge $x$ must be connected to the triangle, say $x=AD$. Then $[x,u]=[x,BC]=0$ and hence $[x,[y,z]]=[[x,z],u]=[u,[x,y]]$ which implies that $[x,z]\ne0$ and $[x,y]\ne0$. Hence there are edges $v=CD,w=BD$ and it follows that $ABCD$ is a $K_4$.  
\end{proof}
Here is an example of a non-representable set-arrangement, where the same technique as above may be used to obtain the structure of the holonomy Lie algebra $L$. It is the "Fano" matroid of the seven non-zero vectors in $\mathbb Z_2^3$ (and hence representable in characteristic 2). There are seven triples of dependent vectors, which may be numbered as follows
$$
(123),\ (145),\ (167),\ (246),\ (257),\ (347),\ (356)
$$
We have an exact sequence $0\to\subalg_L(4567)\to L\to L_{(123)}\to0$. We may define an action of  $L_{(123)}$ onto $\F(4567)$ using the six last triples and checking that the relations operate as zero. This gives that $\subalg_L(4567)$ is free and we get that $L$ is a successive semi-direct product of free Lie algebras on 4,2,1 generators. Dividing $L$ by the ideal generated by $[1,2]$ gives the "non-Fano" holonomy Lie algebra with set-arrangement as above but with (123) deleted. Its series seems difficult to determine.

Here is the method by Lima-Filho and Schenk \cite{lim} to obtain the series for $L_G$ for an arbitrary graph $G$. Consider the subdivision of $G$ above. We have a commutative diagram of surjective maps of Lie algebras,
$$
\begin{CD}
L_G@>>> L_{G_1}\\@VVV @VVV\\
L_{G_2} @>>> L_K
\end{CD}
$$  
which they  proved is a pullback diagram. There is also a pushout diagram of injective maps of Lie algebras,
$$
\begin{CD}
L_K@>>> L_{G_1}\\@VVV @VVV\\
L_{G_2} @>>> L_G
\end{CD}
$$ 
yielding $L_G\cong (L_{G_1}\oplus L_{G_2})/L_K$ (where $L_K$ is considered as the "diagonal" in the direct sum). From this they find their formula by induction. The result is that the series contains the series for the free Lie algebra on $j$ generators, $1/(1-jt)$, with exponent
$$
\sum_{s\ge j}(-1)^{s-j}{s\choose j}\kappa_s
$$ 
where $\kappa_s$ is the number of complete subgraphs of $G$ on $s+1$ vertices. 

\vspace{10pt}
{\bf Question 1:} Is this exponent always positive?

\vspace{10pt}
In \cite{roo} there are examples of "exotic" holonomy Lie algebras, e.g., the LCS series, i.e., the Hilbert series of the enveloping algebra, can be transcendental, disproving a natural question, inspired by the result above for graphs, namely if the holonomy Lie algebra  always can be obtained in a finite number of steps from free Lie algebras by extensions and pullbacks. But the following question might be open.

\vspace{10pt}
{\bf Question 2:}
 Is the global dimension finite for the enveloping algebra of a holonomy Lie algebra $\eta$, i.e., is it true that $\tor_i^{\U(\eta)}(k,k)$ for big enough $i$?

\end{document}